\RequirePackage{amsmath}
\documentclass{llncs}

\usepackage[utf8]{inputenc}
\usepackage[T1]{fontenc}
\usepackage{latexsym}

\usepackage{amssymb}

\usepackage{url}
\usepackage{cite}
\usepackage{hyperref}

\usepackage{graphicx}
\usepackage{tabularx}
\usepackage{array}
\usepackage{mdwmath}
\usepackage{mdwtab}

\usepackage[cmyk]{xcolor}
\usepackage{tikz}

\newcommand{\pc}{\mathbf{P}}

\newcommand{\pp}{\mathbb{P}}

\begin{document}
\setcounter{page}{1}

\title{On Deductive Systems of AC Semantics for Rough Sets}
\author{\textbf{A. Mani}}
\institute{Department of Pure Mathematics\\
University of Calcutta\\
9/1B, Jatin Bagchi Road\\
Kolkata(Calcutta)-700029, India\\
Email: \email{a.mani.cms@gmail.com}\\
Homepage: \url{http://www.logicamani.in}}

\maketitle

\begin{abstract}
Antichain based semantics for general rough sets were introduced recently by the present author. In her paper two different semantics, one for general rough sets and another for general approximation spaces over quasi-equivalence relations, were developed. These semantics are improved and studied further from a lateral algebraic logic perspective in this research. The main results concern the structure of the algebras and deductive systems in the context.  

\end{abstract}

\section{Introduction}

It is well known that sets of rough objects (in various senses) are quasi or partially orderable. Specifically in classical or Pawlak rough sets \cite{PPM2}, the set of roughly equivalent sets has a Quasi Boolean order on it while the set of rough and crisp objects is Boolean ordered.  In the classical semantic domain or classical meta level, associated with general rough sets, the set of crisp and rough objects is quasi or partially orderable. Under minimal assumptions on the nature of these objects, many orders with rough ontology can be  associated - these necessarily have to do with concepts of discernibility. Concepts of rough objects, in these contexts, depend additionally on approximation operators and granulations used. These were part of the motivations of the development of the concept of granular operator spaces by the present author in \cite{AM6999}. 

In quasi or partially ordered sets, sets of mutually incomparable  elements are called \emph{antichains} (for basics see\cite{DP2002,koh}). The possibility of using antichains of rough objects for a possible semantics was mentioned in \cite{AM6000,AM9501,AM3690} by the present author and developed in \cite{AM6999}. The semantics is applicable for a large class of operator based rough sets including specific cases of \textsf{RYS} \cite{AM240} and other general approaches like \cite{CD3}.

The semantics of \cite{AM6999}, as improved in the present paper, is very general, open ended, extendable and optimal for lateral studies.  In the same framework, the machinery for isolation of deductive systems is developed  and studied from a purely algebraic logic point of view. New results on representation are also proved. 

\subsection*{Background}

Let $S$ be any set and $l, u$ be lower and upper approximation operators on $\mathcal{S}\subseteq \wp (S)$ that satisfy monotonicity and $(\forall A\subseteq S)\, A\subseteq A^u$. An element $A\in\mathcal{S}$ will be said to be \emph{lower definite} (resp. \emph{upper definite}) if and only if $A^l = A$ (resp. $A^u = A$) and \emph{definite}, when it is both lower and upper definite. For possible concepts of rough objects \cite{AM6999,AM240} may be consulted. \textsf{Finiteness of $S$ and granular operator spaces, defined below, will be assumed (though not always essential) unless indicated otherwise}.  

Set framework with operators will be used as all considerations will require quasi orders in an essential way. The evolution of the operators need not be induced by a cover or a relation (corresponding to cover or relation based systems respectively), but these would be special cases. The generalization to some rough Y-systems \textsf{RYS} (see \cite{AM240} for definitions), will of course be possible as a result. 

\begin{definition}[\cite{AM6999}]
A \emph{Granular Operator Space} $S$ will be a structure of the form $S\,=\, \left\langle \underline{S}, \mathcal{G}, l , u\right\rangle$ with $\underline{S}$ being a set, $\mathcal{G}$ an \emph{admissible granulation}(defined below) over $S$ and $l, u$ being operators $:\wp(\underline{S})\longmapsto \wp(\underline{S})$ satisfying the following ($\underline{S}$ will be replaced with $S$ if clear from the context. Lower case alphabets will often be used for subsets ):

\begin{align*}
A^l \subseteq A\,\&\,A^{ll} = A^l \,\&\, A^{u} \subset A^{uu}  \\
(A\subseteq B \longrightarrow A^l \subseteq B^l \,\&\,A^u \subseteq B^u)\\
\emptyset^l = \emptyset \,\&\,\emptyset^u = \emptyset \,\&\,\underline{S}^{l}\subseteq S \,\&\, \underline{S}^{u}\subseteq S.
\end{align*}

Here, \textsf{Admissible granulations} are granulations $\mathcal{G}$ that satisfy the following three conditions (Relative \textsf{RYS} \cite{AM240}, $\pc = \subseteq$, $\pp = \subset$) and $t$ is a term operation formed from set operations):

\begin{align*}
(\forall x \exists
y_{1},\ldots y_{r}\in \mathcal{G})\, t(y_{1},\,y_{2}, \ldots \,y_{r})=x^{l} \\
\tag{Weak RA, WRA} \mathrm{and}\: (\forall x)\,(\exists
y_{1},\,\ldots\,y_{r}\in \mathcal{G})\,t(y_{1},\,y_{2}, \ldots \,y_{r}) =
x^{u},\\
\tag{Lower Stability, LS}{(\forall y \in
\mathcal{G})(\forall {x\in \underline{S} })\, ( y\subseteq x\,\longrightarrow\, y \subseteq (x^{l})),}\\
\tag{Full Underlap, FU}{(\forall
x,\,y\in\mathcal{G})(\exists
z\in \underline{S} )\, x\subset z,\,y \subset z\,\&\,z^{l} = z^{u} = z,}
\end{align*}
\end{definition}

In the present context, these conditions mean that every approximation is somehow representable by granules, that granules are lower definite, and that all pairs of distinct granules are contained in definite objects.

On $\wp(\underline{S})$, the relation $\sqsubset$ is defined by \[A \sqsubset B \text{ if and only if } A^l \subseteq B^l \,\&\, A^u \subseteq B^u.\] The rough equality relation on $\wp(\underline{S})$ is defined via $A\approx B \text{ if and only if } A\sqsubset B  \, \&\,B \sqsubset A$. 

Regarding the quotient $\underline{S}|\approx$ as a subset of $\wp(\underline{S})$, the order $\Subset$ will be defined as per \[\alpha \Subset \beta \text{ if and only if } \alpha^l \subseteq \beta^l \,\&\, \alpha^u \subseteq \beta^u.\] Here $\alpha^l$ is being interpreted as the lower approximation of $\alpha$ and so on. $\Subset$ will be referred to as the \emph{basic rough order}.

\begin{definition}
By a \emph{roughly consistent object} will be meant a set of subsets of $\underline{S}$ of the form  $H = \{A ; (\forall B\in H)\,A^l =B^l, A^u = B^u \}$. The set of all roughly consistent objects is partially ordered by the inclusion relation. Relative this maximal roughly consistent objects will be referred to as \emph{rough objects}. By \emph{definite rough objects}, will be meant rough objects of the form $H$ that satisfy 
\[(\forall A \in H) \, A^{ll} = A^l \,\&\, A^{uu} = A^{u}. \] 
\end{definition}

\begin{proposition}
$\Subset$ is a bounded partial order on $\underline{S}|\approx$. 
\end{proposition}

\begin{theorem}
\begin{itemize}
\item {The poset $AC_m (X)$ of all maximum sized antichains of a poset $X$ is a distributive lattice.}
\item {For every finite distributive lattice $L$ and every chain decomposition ${C}$ of $J_L$ (the set of join irreducible elements of $L$), there is a poset $X_C$ such that $L\cong AC_m(X_C)$. }
\end{itemize} 
\end{theorem}
Links to proofs can be found in \cite{AM6999}.

\section{Deductive Systems}\label{cons}

In this section, key aspects of the approach to ternary deductive systems in \cite{IC2003,IC2001} are presented. These are intended as natural generalizations of the concepts of ideals and filters and classes of congruences that can serve as subsets or subalgebras closed under consequence operations or relations (also see \cite{FJ}).

\begin{definition}\label{sys}
Let $\mathbb{S} = \left\langle S, \Sigma \right\rangle$ be an algebra, then the set of term functions over it will be denoted by $\mathbf{T}^\Sigma (\mathbb{S})$ and the set of $r$-ary term functions by $\mathbf{T}^\Sigma_r (\mathbb{S})$. Further let 
\begin{align*}
\tag{0} g\in \mathbf{T}^\Sigma_1 (\mathbb{S}), \;z\in S , \; \tau \subset \mathbf{T}^\Sigma_3 (\mathbb{S}), \\
\tag{1} g(z)\in \Delta \subset S ,\\
\tag{2} (\forall t\in \tau )(a\in \Delta \, \&\,t (a, b, z) \in \Delta \longrightarrow b\in \Delta ), \\
\tag{3} (\forall t \in \tau)( b\in \Delta \longrightarrow t (g(z), b, z) \in \Delta ),
\end{align*}
then $\Delta$ is a $(g, z)-$\emph{$\tau$-deductive system} of $\mathbb{S}$. If further for each $k$-ary operation $f\in \Sigma$ \[ (\forall a_i, b_i \in S)(\&_{i=1}^{k}p(a_i, b_i , z)\in \Delta \longrightarrow p(f(a_1, \ldots , a_k), f(f(a_1, \ldots , b_k),x)\in \Delta),\] then $\Delta$ is said to be compatible.

$\tau$ is said to be a \emph{g-difference system} for $\mathbb{S}$ if $\tau$ is finite and the condition 
\[(\forall t\in \tau) t(a, b, c) = g(c) \text{ if and only if } a = b \text{ holds.}\]
\end{definition}

A variety $\mathcal{V}$ of algebras is regular with respect to a unary term $g$ if and only if for each $S\in \mathcal{V}$,
\[(\forall b\in S)(\forall \sigma, \rho \in con(S))([g(b)]_\sigma = [g(b)]_\rho \longrightarrow \sigma=\rho) .\]
It should be noted that in the above $\tau$ is usually taken to be a finite subset and a variety has a g-difference system if and only if it is regular with respect to $g$.

\begin{proposition}
In the above definition, it is provable that
\[(\forall t \in \tau)( t (g(z), b, z) \in \Delta \longrightarrow  b\in \Delta ).\]
\end{proposition}

\begin{definition}\label{rela}
In the context of Def.\ref{sys}, $\Theta_{Delta, z} $ shall be a relation induced on $S$ by $\tau$ as per the following 
\[(a, b) \in \Theta_{\Delta , z} \text{ if and only if } (\forall t \in \tau)\, t(a, b, z)\in \Delta .\]
\end{definition}

\begin{proposition}
In the context of Def.\ref{rela}, $\Delta = [g(z)]_{\Theta_{\Delta, z}}$.  
\end{proposition}

\begin{proposition}
Let $\tau\subset T_{3}^{\Sigma}(\mathbb{S})$ with the algebra $\mathbb{S} = \left\langle S, \Sigma \right\rangle$, $v\in T_{1}^{\Sigma}(\mathbb{S})$, $e\in S$, $K \subseteq S$ and let $\Theta_{K, e}$ be induced by $\tau$. If $\Theta_{K, e}$ is a reflexive and transitive  relation such that $K = [v(e)]_{Theta_{K, e}}$ , then $K$ is a $(v, e)$- $\tau$-deductive system of $\mathbb{S}$.
\end{proposition}

\begin{theorem}
Let $h$ is a unary term of a variety $\mathcal{V}$ and $\tau$ a h-difference system for $\mathcal{V}$. If $\mathbb{S}\in \mathcal{V}$m $\Theta\in Con(\mathbb{S})$, $z\in S$ and $\Delta = [h(z)]_\Theta$, then $\Theta_{\Delta , z} = \Theta$ and $\Delta$ is a compatible $(h, z)$-$\tau$-deductive system of $\mathbb{S}$. 
\end{theorem}
The converse holds in the following sense:
\begin{theorem}
If $h$ is a unary term of a variety $\mathcal{V}$, $\tau$ is a $h$-difference
system in it, $\mathbb{S}\in\mathcal{V}$, $z\in S$ and if $\Delta$ is a compatible $(h, z)$-$\tau$-deductive system of $\mathbb{S}$, then $\Theta_{\Delta, z} \in Con(\mathbb{S})$ and $\Delta = [g(z)]_{\Theta_{\Delta, z}}$.
\end{theorem}

When $\mathcal{V}$ is regular relative $h$, then $\mathcal{V}$ has a $h$-difference system relative $\tau$ and for each $\mathbb{S}\in \mathcal{V}$, $z\in S$ and $\Delta \subset S$, $\Delta = [h(z)]$ if and only if $\Delta$ is a $(h, z)$- $\tau$-deductive system of $\mathbb{S}$. 

In each case below, $\{t\}$ is a $h$-difference system ($x\oplus y = ((x\wedge y^*)^*\wedge (x^*\wedge y)^*)^*$):
\begin{align*}
\tag{Variety of Groups} h(z) = z \; \&\; t(a, b, c) = a - b +c \\
\tag{Variety of Boolean Algebras} h(z) = z \;\&\; t(a, b, c) = a\oplus b \oplus c \\
\tag{Variety of p-Semilattices} h(z)= z^{**}\;\&\; t(a, b, c) = (a+b)+ c 
\end{align*}

\section{Anti Chains for Representation}\label{var}

In this section, the main algebraic semantics of \cite{AM6999} is summarized, extended to \emph{AC}-algebras and relative properties are studied. It is also proved that the number of maximal antichains required to generate the AC-algebra is rather small.

\begin{definition}
$\mathbb{A,B} \in \underline{S} |\approx$, will be said to be \emph{simply independent} (in symbols $\Xi (\mathbb{A, B})$)if and only if 
\[\neg (\mathbb{A} \Subset \mathbb{B}) \text{ and } \neg (\mathbb{B} \Subset \mathbb{A}) .\] 

A subset $\alpha \subseteq \underline{S}|\approx$ will be said to be \emph{simply independent} if and only if 
\[(\forall \mathbb{A, B} \in \alpha)\, \Xi(\mathbb{A, B})\,\vee \, (\mathbb{A = B}).\]

The set of all simply independent subsets shall be denoted by $\mathcal{SY}(S)$. 

A \emph{maximal simply independent subset}, shall be a simply independent subset that is not properly contained in any other simply independent subset. The set of maximal simply independent subsets will be denoted by $\mathcal{SY}_m (S)$. On the set $\mathcal{SY}_m (S)$, $\ll$ will be the relation defined by 
\[\alpha\ll\beta \text{ if and only if } (\forall \mathbb{A}\in \alpha)(\exists \mathbb{B}\in\beta)\, \mathbb{A \Subset B}. \]
\end{definition}

\begin{theorem}
$\left\langle\mathcal{SY}_m (S), \ll \right\rangle$ is a distributive lattice.
\end{theorem}

Analogous to the above, it is possible to define essentially the same order on the set of maximal antichains of $\underline{S}|\approx$ denoted by $\mathfrak{S}$ with the $\Subset$ order. This order will be denoted by $\lessdot$ - this may also be seen to be induced by maximal ideals. 

\begin{theorem}\label{par}
If $\alpha = \{\mathbb{A}_1, \mathbb{A}_2, \ldots , \mathbb{A}_n, \dots \} \in \mathfrak{S}$, and if $L$ is defined by \[L(\alpha) = \{\mathbb{B}_1, \mathbb{B}_2, \ldots , \mathbb{B}_n , \ldots \}\] with 
$X\in \mathbb{B}_i$ if and only if $X^l = \mathbb{A}_i^{ll}= \mathbb{B}_i^l$ and $X^u = \mathbb{A}_i^{lu} = \mathbb{B}_i^u$, then $L$ is a partial operation in general. 
\end{theorem}

\begin{definition}
Let $\chi(\alpha \cap \beta) = \{\xi ;\; \xi \text{ is a maximal antichain } \& \, \alpha \cap \beta \subseteq \xi \}$ be the set of all possible extensions of $\alpha\cap\beta$. The function $\delta : \mathfrak{S}^2 \longmapsto  \mathfrak{S}$ corresponding to \emph{extension under cognitive dissonance} will be defined as per $\delta(\alpha, \beta) \,\in\, \chi(\alpha\cap\beta)$ and (LST means \emph{maximal subject to})
 
\[\delta (\alpha, \beta)  \,=\, \left\{
\begin{array}{ll}
\xi, & \text{ if } \xi\cap\beta \text{ is a maximum subject to } \xi\neq\beta\text{ and } \xi \text{ is unique },\\
\xi, & \text{ if } \xi\cap\beta \,\&\, \xi\cap\alpha\, \text{ are LST } \xi\neq\beta, \alpha  \text{ and } \xi \text{ is unique },\\
\beta, & \text{ if } \xi\cap\beta \,\&\, \xi\cap\alpha \text{ are LST } \&\,\xi\neq\beta, \alpha \text{ but } \xi \text{ is not unique },\\
\beta, &  \text{ if } \chi(\alpha\cap\beta) = \{\alpha, \beta\}. 
\end{array}
\right.\]
\end{definition}

\begin{definition}
In the context of the above definition, the function $\varrho : \mathfrak{S}^2 \longmapsto  \mathfrak{S}$ corresponding to \emph{radical extension } will be defined as per $\varrho(\alpha, \beta) \,\in\, \chi(\alpha\cap\beta)$ and (MST means \emph{minimal subject to})
\[\varrho (\alpha, \beta)  \,=\, \left\{
\begin{array}{ll}
\xi, & \text{ if } \xi\cap\beta \text{ is a minimum under } \xi\neq\beta\text{ and } \xi \text{ is unique },\\
\xi, & \text{ if } \xi\cap\beta \,\&\, \xi\cap\alpha \text{ are MST } \xi\neq\beta, \alpha \text{ and } \xi \text{ is unique },\\
\alpha, & \text{ if } (\exists \xi)\, \xi\cap\beta \,\&\, \xi\cap\alpha \text{ are MST } \xi\neq\beta, \alpha \text{ but } \xi \text{ is not unique },\\
\alpha, &  \text{ if } \chi(\alpha\cap\beta) = \{\alpha, \beta\}. 
\end{array}
\right.\]
\end{definition}

\begin{theorem}
The operations $\varrho ,\,\delta $ satisfy all of the following: 
\begin{enumerate}
\item {$\varrho ,\,\delta $ are groupoidal operations,}
\item {$\varrho (\alpha, \alpha ) = \alpha$,}
\item {$\delta (\alpha, \alpha ) = \alpha$,}
\item {$ \delta(\alpha, \beta) \cap \beta \subseteq \delta(\delta(\alpha, \beta ),\beta) \cap \beta  $,}
\item {$\delta (\delta(\alpha,\beta),\beta) = \delta(\alpha, \beta) $}
\item {$ \varrho(\varrho(\alpha, \beta ),\beta) \cap \beta \subseteq\varrho(\alpha, \beta) \cap \beta  $.}
\end{enumerate}
\end{theorem}

In general, a number of possibilities (potential non-implications) like the following are satisfied by the algebra:  $\alpha\lessdot \beta \,\& \, \alpha\lessdot \gamma \,\nrightarrow \, \alpha \lessdot \delta(\beta, \gamma)$. Given better properties of $l$ and $u$, interesting operators can be induced on maximal antichains towards improving the properties of $\varrho $ and $\delta$. The key constraint hindering the definition of total $l, u$ induced operations can be avoided in the following way:

\begin{definition}
In the context of Thm \ref{par}, operations $\square, \Diamond$ can be defined as follows:
\begin{itemize}
\item {Given $\alpha = \{\mathbb{A}_1, \mathbb{A}_2, \ldots , \mathbb{A}_n, \dots \} \in \mathfrak{S}$, form the set\\ $\gamma(\alpha) = \{\mathbb{A}_1^l, \mathbb{A}_2^l, \ldots , \mathbb{A}_n^l, \dots \}$. If this is an antichain, then $\alpha$ would be said to be \emph{lower pure}.}
\item {Form the set of all relatively maximal antichains $\gamma_+(\alpha)$ from $\gamma(\alpha)$. }
\item {Form all maximal antichains $\gamma_{*}(\alpha)$ containing elements of $\gamma_+(\alpha)$ and set $\square (\alpha)= \wedge\gamma_*(\alpha)$  }
\item {For $\Diamond$, set $\pi(\alpha) = \{\mathbb{A}_1^u, \mathbb{A}_2^u, \ldots , \mathbb{A}_n^u, \dots \}$. If this is an antichain, then $\alpha$ would be said to be \emph{upper pure}.}
\item {Form the set of all relatively maximal antichains $\pi_+(\alpha)$ from $\pi(\alpha)$}
\item {Form all maximal antichains $\pi_{*}(\alpha)$ containing elements of $\pi_+(\alpha)$ and set $\Diamond(\alpha) = \vee\pi_*(\alpha)$}
\end{itemize}
\end{definition}

\begin{theorem}
In the context of the above definition, the following hold:
\begin{align*}
\alpha \lessdot \beta \longrightarrow \square (\alpha) \lessdot \square (\beta) \,\&\, \Diamond (\alpha) \lessdot \Diamond (\beta)\\
\square (\alpha) \lessdot \alpha \lessdot \Diamond (\alpha), \; \square (0) = 0 \,\&\, \Diamond (1) = 1
\end{align*}
\end{theorem}

Based on the above properties, the following algebra can be defined.

\begin{definition}
By a \emph{Concrete AC } algebra (AC -algebra)  will be meant an algebra of the form $\left\langle\mathfrak{S}, \varrho , \delta , \vee , \wedge, \square , \Diamond ,  0, 1  \right\rangle$  associated with a granular operator space $S$ satisfying all of the following:
\begin{itemize}
\item {$\left\langle\mathfrak{S}, \vee , \wedge  \right\rangle$ is a bounded distributive lattice derived from a granular operator space as in the above.}
\item {$\varrho, \delta, \square , \Diamond $ are as defined above}
\end{itemize}
\end{definition}

The following concepts of ideals and filters are of interest as deductive systems in a natural sense and relate to ideas of rough consequence (detailed investigation will appear separately).

\begin{definition}
By a \emph{LD-ideal} (resp. LD-filter)) $K$ of an AC-algebra $\mathfrak{S}$ will be meant a lattice ideal (resp. filter) that satisfies:\[(\forall \alpha \in K)\, \square(\alpha), \Diamond (\alpha) \in K\]
By a \emph{VE-ideal} (resp. VE-filter)) $K$ of an AC-algebra $\mathfrak{S}$ will be meant a lattice ideal (resp. filter) that satisfies:\[(\forall \xi \in \mathfrak{S})(\forall \alpha \in K)\, \varrho(\xi, \alpha), \delta (\xi , \alpha) \in K\]
\end{definition}

\begin{proposition}
Every VE filter is closed under $\Diamond$ 
\end{proposition}

\subsection{Generating AC-Algebras}

Now it will be shown below that specific subsets of AC-algebras suffice to generate the algebra itself and that the axioms satisfied by the granulation affect the generation process and properties of AC-algebras and forgetful variants thereof. 

An element $x\in \mathfrak{S}$ will be said to be \emph{meet irreducible} (resp. join irreducible) if and only if $\wedge\{x_i\} = x \longrightarrow (\exists i)\, x_i = x $ (resp. $\vee\{x_i\} = x \longrightarrow (\exists i)\, x_i = x $). Let $W(S), \, J(S)$ be the set of meet and join irreducible elements of $\mathfrak{S}$ and let $l(\mathfrak{S})$ be the length of the distributive lattice.

\begin{theorem}
All of the following hold:
\begin{itemize}
\item {$\left(\mathfrak{S}, \vee, \wedge, 0, 1\right)$ is a isomorphic to the lattice of principal ideals of the poset of join irreducibles.}
\item {$l(\mathfrak{S}) = \# (J(S)) = \# (W(S)) $}
\item {$J(S)$ is not necessarily the set of sets of maximal antichains of granules in general. }
\item {When $\mathcal{G}$ satisfies mereological atomicity that is $(\forall a\in
\mathcal{G})(\forall b \in S)(\pc ba,\,a^{l} = a^{u} = a \longrightarrow a = b)$, and all approximations are unions of granules, then elements of $J(S)$ are proper subsets of $\mathcal{G}$.}
\item {In the previous context, $W(S)$ must necessarily consist of two subsets of $S$ that are definite and are not parts of each other.}
\end{itemize}
\end{theorem}

\begin{proof}
\begin{itemize}
\item {The first assertion is a well known.}
\item {Since the lattice is distributive and finite, its length must be equal to the number of elements in $J(S)$ and $W(S)$. For a proof see \cite{MG1980}. }
\item {Under the minimal assumptions on $\mathcal{G}$, it is possible for definite elements to be properly included in granules as in esoteric or prototransitive rough sets \cite{AM24,AM3690}. These provide the required counterexamples. }
\item {The rest of the assertions follows from the nature of maximal antichains and the constructive nature of approximations.}
\end{itemize} 
\qed
\end{proof}

\begin{theorem}
In the context of the previous theorem if $R(\Diamond), \, R(\square) $ are the ranges of the operations $\Diamond, \square$ respectively, then these have a induced lattice order on them. Denoting the associated lattice operations by $\curlyvee , \curlywedge$ on $R(\Diamond )$, it can be shown that 
\begin{itemize}
\item {$R(\Diamond)$ can be reconstructed from $J(R(\Diamond)) \cup W(R(\Diamond))$.}
\item {$R(\square)$ can be reconstructed from $J(R(\square)) \cup W(R(\square))$.}
\item {When $\mathcal{G}$ satisfies mereological atomicity and absolute crispness (i.e. $(\forall x\in \mathcal{G})\, x^l = x^u = x$), then $R(\Diamond)$ are lattices which are constructible from two sets $A$, $C$ (with $A = \{\mathcal{G}\cup \{g_1\cup g_2\}^u\setminus \{g_1, g_2  \}; \, g_1, g_2 \in \mathcal{G} \}$ and $C$ being the set of two element maximal antichains formed by sets that are upper approximations of other sets).   }
\end{itemize}
\end{theorem}
\begin{proof}
It is clear that $R(\Diamond)$ is a lattice in the induced order with $J(R(\Diamond))$ and $W(R(\Diamond))$ being the partially ordered sets of join and meet irreducible elements respectively. This holds because the lattice is finite.

The reconstruction of the lattice can be done through the following steps:
\begin{itemize}
\item {Let $Z= J(R(\Diamond)) \cup W(R(\Diamond))$. This is a partially ordered set in the order induced from $R(\Diamond)$.  }
\item {For $b\in J(R(\Diamond))$ and $ a\in W(R(\Diamond))$, let $b\prec a$ if and only if $a \neq b$ in $R(\Diamond)$.}
\item {On the new poset $Z$ with $\prec$, form sets including elements of $W(R(\Diamond))$ connected to it.   }
\item {The set of union of all such sets including empty set ordered by inclusion would be isomorphic to the original lattice. \cite{MG1980}}
\item {Under additional assumptions on $\mathcal{G}$, the structure of $Z$ can be described further.}
\end{itemize}

When the granulation satisfies the properties of crispness and mereological atomicity, then $A= J(R(\Diamond))$ and $C= W(R(\Diamond))$. So the third part holds as well.
\qed
\end{proof}

The results motivate this concept of purity: A maximal antichain will be said to \emph{pure} if and only if it is both lower and upper pure.

\subsection{Enhancing the Anti Chain Based Representation}

An integration of the orders on sets of maximal antichains or antichains and the representation of rough objects and possible orders among them leads to interesting multiple orders on the resulting structure. A major problem is that of defining the orders or partials thereof in the first place among the various possibilities.

\begin{definition}
By the \emph{rough interpretation of an antichain} will be meant the sequence of pairs obtained by substituting objects in the rough domain in place of objects in the classical perspective. Thus if $\alpha= \{a_1, a_2. \ldots , a_p\}$ is a antichain, then its rough interpretation would be ($\pi(a_i) = (a_i^l, a_i^u)$ for each $i$)\[\underline{\alpha} =\{\pi(a_1), \pi(a_2), \ldots , \pi (a_p)\}.\] 
\end{definition}

\begin{proposition}
It is possible that some rough objects are not representable by maximal antichains. 
\end{proposition}

\begin{proof}
Suppose the objects represented by the pairs $(a, b)$ and $(e, f)$ are such that $a=e$ and $b\subset f$, then it is clear that any maximal antichain containing $(e, f)$ cannot contain any element from $\{x\,: x^l =a\,\&\,x^u = b \}$. This situation can happen, for example, in the models of proto transitive rough sets \cite{AM9501,AM6000}. Concrete counterexamples can be found in the same paper.
\qed
\end{proof}

\begin{definition}
A set of maximal antichains $V$ will be said to be \emph{fluent} if and only if  
$(\forall x )(\exists \alpha \in V)(\exists (a, b) \in \alpha ) \, x^l = a \,\&\, x^u = b$. 

It will be said to be \emph{well fluent} if and only if it is fluent and no proper subset of it is fluent.
\end{definition}
A related problem is of finding conditions on $\mathcal{G}$, that ensure that $V$ is fluent.

\section{Ternary Deduction Terms}

Since AC-algebras are distributive lattices with additional operations, a natural strategy should be to consider terms similar to Boolean algebras and p-Semilattices. For isolating deductive systems in the sense of Sec.\ref{cons}, a strategy can be through complementation-like operations. This motivates the following definition:

\begin{definition}
In a AC-algebra $\mathfrak{S}$, if an antichain $\alpha = (X_1, X_2, \ldots , X_n )$, then some possible general complements on the schema $\alpha^c =\mathfrak{H} (X_1^c, X_2^c, \ldots , X_n^c ) $ are as follows:
\begin{align*}
\tag{Class A} X_i^{*} = \{w ;\,(\forall a\in X_i)\,\neg \pc a w \,\&\, \neg \pc w a \} \\
\tag{Light} X_i^{\#} = \{w ;\,(\forall a\in X_i)\, \neg a^l = w^l \text{ or } \neg a^u = w^u \} \\ 
\tag{UU} X_i^{\flat} = \{w ;\,(\forall a\in X_i) \,\neg a^l = w^l \text{ or } \neg a^{uu} = w^{uu}\} 
\end{align*}

$\mathfrak{H}$ is intended to mean the maximal antichain containing the set if that is definable. 
\end{definition}

Note that under additional assumptions (similarity spaces), the light complementation is similar to the preclusivity operation in \cite{CC} for Quasi BZ-lattice or Heyting-Wajsburg semantics and variants.

The above operations on $\alpha$ are partial in general, but can be made total with the help of an additional order on $\alpha$ and the following procedure:
\begin{enumerate}
\item{Let $\alpha = \{X_1 , X_2, \ldots , X_n \}$ be a finite sequence, }
\item{Form $\alpha^c$ and split into longest ACs in sequence,}
\item{Form maximal ACs containing each AC in sequence}
\item{Join resulting maximal ACs.}
\end{enumerate}

\begin{proposition}
Every general complement defined by the above procedure is well defined.
\end{proposition}

\begin{proof}
\begin{itemize}
\item {Suppose $\{X_1^c, X_2^c \}, \{X_3^c,\ldots X_n^c \}$ form antichains, but $\{X_1^c, X_2^c ,X_3^c\}$ is not an antichain.}
\item {Then form the maximal antichains $\eta_1, \ldots , \eta_p$ containing either of the two antichains.}
\item {The join of this finite set of maximal antichains is uniquely defined. By induction, it follows that the operations are well defined.}
\end{itemize}
\qed
\end{proof}

\subsection{Translations}

As per the approach of Sec \ref{cons}, possible definitions of translations are as follows:
\begin{definition}
A \emph{translation} in a AC-algebra $\mathfrak{S}$ is a $\sigma : \mathfrak{S}\longmapsto \mathfrak{S} $ that is defined in one of the following ways (for a fixed $a\in \mathfrak{S}$):
\begin{align*}
\sigma_{\theta} (x) = \theta(a,  x) \; ; \theta \in \{\vee, \wedge, \varrho, \delta \} \\
\sigma_{\mu} (x) = \mu (x) \; ; \mu \in \{\square , \Diamond \}\\
\sigma_{t} (x) = (x\oplus a) \oplus b\text{ for fixed a, b}  \\
\sigma_{t+} (x) = (a\oplus b) \oplus x\text{ for fixed a, b}  
\end{align*}
\end{definition}

\begin{theorem}
\begin{align*}
\sigma_\vee (0) = a = \sigma_\vee (a  \,; \sigma_\vee (1) =1\\
Ran(\sigma_{\vee}) \text{ is the principal filter generated by } a\\
Ran(\sigma_{\wedge}) \text{ is the principal ideal generated by } a\\
x\lessdot w \longrightarrow \sigma_{\vee}(x)\lessdot \sigma_{\vee}(w)\,\&\, \sigma_{\wedge}(x)\lessdot \sigma_{\wedge}(w)
\end{align*}
\end{theorem}
\begin{proof}
\begin{itemize}
\item { Let $\mathbb{F}(a)$ be the principal lattice filter generated by $a$.}
\item {If $a\lessdot w$, then $a\vee w = w = \sigma_\vee(w)$. So $w\in Ran (\sigma_\vee)$.}
\item {$\sigma_\vee (x)\wedge \sigma_\vee (w) = (a\vee x) \wedge (a\vee w) = a\vee (x\wedge w) =\sigma_\vee (x\wedge w)$.}
\item {So if $x, w \in Ran (\sigma_\vee)$, then $x\wedge w, x\vee w \in Ran (\sigma_\vee)$}
\item {Simly it is provable that $Ran(\sigma_\wedge)$ is the principal ideal generated by $a$.}
\end{itemize}
\qed
\end{proof}

\subsection{Ternary Terms and Deductive Systems}

Possible ternary terms that can cohere with the assumptions of the semantics include the following
$t(a, b, z) = a\wedge b\wedge z$, $t(a, b, z) = a\oplus b\oplus z$ ($\oplus $ being as indicated earlier) and $t(a, b, z) = \square (a\wedge b) \wedge z $. These have admissible deductive systems associated. Further under some conditions on granularity, the distributive lattice structure associated with $\mathfrak{S}$ becomes pseudo complemented.

\begin{theorem}
If $t(a, b, z) = a\wedge b \wedge z$, $\tau= \{t\}$, $z\in H$, $h(x)= x$ $\sigma(x) = x\wedge z$ and if $H$ is a ternary $\tau$-deduction system at $z$, then it suffices that $H$ be an filter. 
\end{theorem}
\begin{proof}
All of the following must hold:
\begin{itemize}
\item {If $a\in H$, $t(z, a, z) = a\wedge z\in H$}
\item {If $t(a, b, z) \in H$, then $t(\sigma(a), \sigma(b), z) = t(a, b, z) \in H$}
\item {If $a, t(a, b, z) \in H$ then $t(a, b, z) = (a\wedge z)\wedge b\in H$. But $H$ is a filter, so $b\in H$.}
\end{itemize} 
\qed
\end{proof}
\begin{theorem}
If $t(a, b, z) = (a\vee (\square b) )\wedge z$, $\tau= \{t\}$, $z\in H$, $h(x)= x$ $\sigma(x) = x\wedge z$ and if $H$ is a ternary $\tau$-deduction system at $z$, then it suffices that $H$ be a principal LD-filter generated by $z$. 
\end{theorem}
\begin{proof}
All of the following must hold:
\begin{itemize}
\item {If $a\in H$, $t(z, a, z) = (z\vee (\square a) )\wedge z\in H$ because $(z\vee (\square a) ) \in H$.}
\item {If $t(a, b, z) \in H$, then $t(\sigma(a), \sigma(b), z) = t((a\wedge z), (b\wedge z), z) = ((a\wedge z)\vee \square (b\wedge z))\wedge z \in H$}
\item {If $a, t(a, b, z) \in H$ then $t(a, b, z) = (a\vee \square(b))\wedge z = (a\wedge z)\vee (\square(b)\wedge z) \in H$. But $H$ is a LD-filter, so $a\vee \square (b) \in H$. This implies $\square(b) \in H$, which in turn yields $b\in H$.}
\end{itemize} 
\qed
\end{proof}

In the above two theorems, the conditions on $H$ can be weakened considerably. The converse questions are also of interest.

The existence of pseudo complements can also help in defining ternary terms that determine deductive systems (or subsets closed under consequence). In general, pseudo complementation  $\circledast $ is a partial unary operation on $\mathfrak{S}$ that is defined by $ x^\circledast =\max \{a \,;\, a\wedge x = 0 \}$ (if the greatest element exists).

There is no one answer to the question of existence as it depends on the granularity assumptions of representation and stability of granules. The following result guarantees pseudo complementation (in the literature, there is no universal approach - it has always been the case that in some case they exist):

\begin{theorem}
In the context of AC-algebras, if the granulation satisfies mereological atomicity and absolute crispness, then a pseudo complementation is definable. 
\end{theorem}
\begin{proof}
Under the conditions on the granulation, it is possible to form the rough interpretation of each antichain.
Moreover the granules can be moved in every case to construct the pseudo complement. The inductive steps in this proof have been omitted.  
\end{proof}

\section*{Concluding Remarks}
In this research, the problem of finding deductive systems in the context of antichain based semantics for general rough sets has been explored and key results have been proved by the present author.  The lateral approach used by her is justified by the wide variety of possible concepts of rough consequence in the general setting. In a forthcoming paper, the framework of granular operator spaces has been expanded with definable parthood relations and semantics has been considered through counting strategies. All this will be explored in greater detail in future work. 

\bibliographystyle{splncs.bst}
\bibliography{biblioam07052016.bib}
\end{document}